\documentclass[12pt]{amsart} 

\usepackage{amscd} 
\usepackage{amsmath} 
\usepackage{amssymb} 
\usepackage{amsthm}
\usepackage{bigdelim}
\usepackage{color} 
\usepackage{enumerate}
\usepackage{mathrsfs}
\usepackage{multirow}
\usepackage[all]{xy} 
\def\e{\varepsilon}

\def\o{{\omega}}
\def\O{{\mathcal{O}}}
\def\P{{\mathbb{P}}}
\def\Q{{\mathbb{Q}}} 
 
\def\Z{{\mathbb{Z}}}

\newtheorem{thm}{Theorem}[section] 
\newtheorem{cor}[thm]{Corollary}
\newtheorem{prop}[thm]{Proposition}

\newtheorem{lem}[thm]{Lemma}
\theoremstyle{definition} 
\newtheorem{defn}[thm]{Definition}

\newtheorem{eg}[thm]{Example} 
\theoremstyle{remark}
\newtheorem{rem}[thm]{Remark}

\newtheorem{ques}[thm]{Question}

\newtheorem{notation}[thm]{Notation}

\newtheorem*{ack}{Acknowledgements}

\title{Nef anti-canonical divisors and rationally connected fibrations}
\author{Sho Ejiri}\author{Yoshinori Gongyo}
\address{Department of Mathematics, Graduate School of Science, Osaka University Toyonaka, Osaka 560-0043, Japan.}
\email{s-ejiri@cr.math.sci.osaka-u.ac.jp}
\address{Graduate School of Mathematical Sciences, the University of Tokyo, 3-8-1 Komaba, Meguro-ku, Tokyo 153-8914, Japan.}
\email{gongyo@ms.u-tokyo.ac.jp}

\date{\today}

\begin{document}
\tolerance = 9999
\maketitle
\begin{abstract}
We study the Iitaka--Kodaira dimension of nef relative anti-canonical divisors. 
As a consequence, we prove that given a complex projective variety with klt singularities, 
if the anti-canonical divisor is nef, then the dimension of a general fibre 
of the maximal rationally connected fibration is at least the Iitaka--Kodaira dimension of the anti-canonical divisor.
\end{abstract}
\section{Introduction} \label{section:intro} 
The positivity of the anti-canonical class of a projective variety is an important notion 
that enables us to find certain geometric features of the variety.
In characteristic zero, 
Boucksom--Demailly--P\v{a}un--Peternell~\cite{BDPP} proved 
that a projective manifold is uniruled if its canonical class is not pseudo-effective. 
Koll\'ar--Miyaoka--Mori~\cite{KMM92} showed that 
a Fano manifold is rationally connected (see also~\cite{Cam92}).
This result was generalised by Hacon--M\textsuperscript cKernan~\cite{HM07} and Zhang~\cite{Zhang-rc}
to the case of klt log Fano pairs, that is, projective klt pairs with ample anti-log canonical class. 
In view of these results, it is natural to ask if there are other positivity conditions of anti-canonical classes
from which we can derive geometric consequences similar to those above.

In this article, we study the Iitaka--Kodaira dimension of nef (relative) anti-canonical divisors, 
and prove the following theorem: 
\begin{thm}\label{thm-intro:hm-inequality}
Let $(X,\Delta)$ be a projective klt pair over a field of characteristic zero 
and $r: X \dashrightarrow W$ the maximal rationally chain connected fibration.
Suppose that $-(K_X+\Delta)$ is nef. 
Then $\kappa(X, -(K_X+\Delta)) \leq \kappa (F, -(K_F+\Delta_F))$. 
Here, $F$ is a general fibre of $r$ and $K_F+\Delta_F=(K_X+\Delta)|_{F}$. 
In particular, it holds that 
\begin{align}
\kappa(X, -(K_X+\Delta)) \leq \mathrm{dim}\,X-\mathrm{dim}\,W. \label{inequality-intro:hm}
\end{align}
\end{thm}
This theorem can be thought of as a generalisation of 
the above result of Hacon--M\textsuperscript cKernan and  Zhang. 
Furthermore, Theorem \ref{thm-intro:hm-inequality} answers a question posed in~\cite{HM05} 
that asks whether inequality (\ref{inequality-intro:hm}) holds or not.
%
 
To prove Theorem \ref{thm-intro:hm-inequality}, we establish an injectivity theorem (Theorem~\ref{thm:inj}). 
A simplified version of the injectivity theorem states the following:
\begin{thm}\label{thm-intro:ejiri-inequality} \samepage
Let $k$ be an algebraically closed field of characteristic zero $($resp. $p>0$$)$. 
Let $f:X\to Y$ be a surjective morphism between smooth projective varieties over $k$ such that $f_*\O_X\cong\O_Y$, 
and let $\Delta$ be an effective $\mathbb Q$-divisor $($resp. $\mathbb Z_{(p)}$$)$ on~$X$. 
Take a general fibre $F$ of $f$ and suppose that $(F,\Delta|_F)$ is klt $($resp. strongly $F$-regular$)$. 
Set $L=-(K_{X/Y}+\Delta)$. If $L$ is nef, then for each $m\in\mathbb Z_{>0}$, the morphism
$$H^0(X,\O_X(\lfloor mL\rfloor))\to H^0(F,\O_F(\lfloor mL|_F\rfloor))$$
induced by restriction is injective. 
In particular, $\kappa(X,L)\le\kappa(F,L|_{F})$ holds. 
\end{thm}
This theorem answers another question posed in \cite{Eji16} 
that asks if the inequality in the statement of the theorem holds.
 
Theorem~\ref{thm-intro:ejiri-inequality} is actually proved in the case where $f$ is an almost holomorphic map, 
that is, a rational map with proper general fibres (Theorem~\ref{thm:inj}). 
Note that the maximal rationally chain connected fibration of a projective variety is almost holomorphic. 
Under the assumption of Theorem~\ref{thm-intro:hm-inequality}, 
as explained in the proof of Corollary~\ref{cor:hm-inequality}, 
we may assume that $K_W$ is $\Q$-linearly trivial. 
Therefore, the relative canonical class can be identified with the absolute one, 
and we can apply the general version of Theorem~\ref{thm-intro:ejiri-inequality}. 
 
Let us explain briefly the strategy of the proof of Theorem~\ref{thm-intro:ejiri-inequality}. 
The main ingredient of the proof is the so-called weak positivity theorem 
developed by several people including Fujita, Kawamata, and Viehweg~\cite{Fuj78,kawamata,viehweg}. 
We employ the log version of this theorem established in~\cite{Cam04,Eji15,Fuj17,Pat14}, 
which concerns the positivity of the direct image sheaves~$f_*\O_X(m(K_{X/Y}+\Delta))$ 
of log pluricanonical bundles~$\O_X(m(K_{X/Y}+\Delta))$. 
Using this, we first show that given a $\mathbb Q$-divisor~$D$ on $Y$, 
if $-(K_X+\Delta+g^*D)$ is nef, then $-(K_Y+D)$ is pseudo-effective (Theorem~\ref{thm:sp}). 
This can be viewed as a generalisation of Chen--Zhang's theorem~\cite{CZ13}
that answers Demailly--Peternell--Schneider's question. 
Next, we prove a variant of the above result, 
assuming the existence of an effective $\Q$-Cartier divisor~$\Gamma\equiv-(K_X+\Delta+g^*D)$ on $X$ (Proposition~\ref{prop:p.e.}). 
In this step, the pair $(F,\Delta|_F)$ is required to be klt or strongly $F$-regular, 
because we replace $\Delta$ by $\Delta+\e\Gamma$ for some small $\e>0$. 
Finally, applying the above variant, we find that the support of $\Gamma$ cannot contain $F$, from which the assertion follows. 
 
\begin{notation} \label{notation:intro}
Let $k$ be an algebraically closed field. 
By a \textit{variety} we mean an integral separated scheme of finite type over $k$.
For a prime $p$, let $\Z_{(p)}$ denote the localisation of $\Z$ at $(p)=p\Z$.
A $\mathbb Q$-Weil divisor $D$ on a normal variety is said to be \textit{$\mathbb Z_{(p)}$-Weil} (resp. \textit{$\mathbb Z_{(p)}$-Cartier}) 
if there exists an integer $m\in\mathbb Z\setminus p\mathbb Z$ such that $mD$ is integral (resp. Cartier). 
\end{notation} 

\begin{ack}
The first author was supported by JSPS KAKENHI $\#$15J09117.
He wishes to express his gratitude to Professor Shunsuke Takagi, Doctors Takeru Fukuoka, and Akihiro Kanemitsu for helpful comments and discussions. 
The second author was partially supported by JSPS KAKENHI $\#$26707002, 15H03611, 16H02141, and  17H02831. Moreover he is grateful to Professors Hiromichi Takagi, Paolo Cascini, and Doctor Kazunori Yasutake for having  nice discussions and giving  suggestions about Theorem \ref{thm-intro:hm-inequality} from his master studies.  The authors would like to thank Junyan Cao for answering  to our question  related to Question \ref{intro:regularity}. They also thank Professors Osamu Fujino and Shin-ichi Matsumura for stimulating discussions and giving useful  comments. Moreover they would like to thank Doctor Chen Jiang  for discussing Remark \ref{rem:cz}.
\end{ack}

\section{Preliminaries}
In this section, we review the basic terminology of the singularities of pairs and the positivity of coherent sheaves. 
\subsection{Singularities of pairs}
We first recall several singularities that appear in the minimal model programme. 
\begin{defn}[{cf.~\cite[Definition 2.34]{KM}, \cite[Remark 4.2]{schsmith-logfano}}]\label{sing of pairs}
Let $X$ be a normal variety over an algebraically closed field
and $\Delta$ an effective $\mathbb{Q}$-Weil divisor on $X$ such that $K_X+\Delta$ is $\mathbb{Q}$-Cartier. 
Let $\pi: \widetilde{X} \to X$ be a birational morphism from a normal variety $\widetilde{X}$.
Then we can write 
$$K_{\widetilde{X}}=\pi^*(K_X+\Delta)+\sum_{E}a(E, X, \Delta) E,$$ 
where $E$ runs through all the distinct prime divisors on $\widetilde{X}$ and the $a(E, X, \Delta)$ are rational numbers. 
We say that the pair $(X, \Delta)$ is \textit{log canonical} (resp. \textit{klt}) 
if $a(E, X, \Delta) \ge -1$ (resp. $a(E, X, \Delta) >-1$) for every prime divisor $E$ over $X$.  
If $\Delta=0$, we simply say that $X$ has only log canonical singularities (resp. log terminal  singularities). 
Assume that $\Delta=0$. 
We say that $X$ has only \textit{canonical singularities} (resp. \textit{terminal singularities}) 
if $a(E, X, 0) \ge 0$ (resp. $a(E, X, 0) >0$) for every exceptional prime divisor $E$ over $X$. 
\end{defn}
Next, let us discuss singularities of pairs in positive characteristic. 
We recall two notions of singularities defined in terms of splittings of the Frobenius morphisms. 
\begin{defn}\label{defi-gFreg} 
Let $X$ be a normal \textit{affine} variety defined over an algebraically closed field $k$ of characteristic $p>0$ 
and $\Delta$ an effective $\Q$-divisor on $X$. 
\begin{enumerate}[(i)]
\item(\cite{HR,Sch08})
We say that $(X, \Delta)$ is \textit{sharply $F$-pure} if there exists an integer $e \in \mathbb{Z}_{>0}$ for which the composition map 
$$\mathcal{O}_X \to F^e_*\mathcal{O}_X \hookrightarrow F^e_*\mathcal{O}_X(\lceil (p^e-1) \Delta \rceil)$$
of the $e$-times iterated Frobenius map $\mathcal{O}_X \to F^e_*\mathcal{O}_X$ 
with a natural inclusion $F^e_*\mathcal{O}_X \hookrightarrow F^e_*\mathcal{O}_X(\lceil (p^e-1) \Delta \rceil)$
splits as an $\mathcal{O}_X$-module homomorphism. 
\item(\cite{HH,HW02})
We say that $(X, \Delta)$ is \textit{strongly $F$-regular} if for every effective divisor $D$ on $X$, 
there exists an integer $e \in \mathbb{Z}_{>0}$ such that the composition map 
$$\mathcal{O}_X \to F^e_*\mathcal{O}_X \hookrightarrow F^e_*\mathcal{O}_X(\lceil (p^e-1) \Delta \rceil+D)$$
of the $e$-times iterated Frobenius map $\mathcal{O}_X \to F^e_*\mathcal{O}_X$ with a natural inclusion $F^e_*\mathcal{O}_X \hookrightarrow F^e_*\mathcal{O}_X(\lceil (p^e-1) \Delta \rceil+D)$
splits as an $\mathcal{O}_X$-module homomorphism. 
\end{enumerate}
For a variety $X$ and an effective $\mathbb Q$-divisor $\Delta$ on $X$, 
we say that $(X, \Delta)$ is \textit{sharply $F$-pure} 
if there exists an  affine open covering $\{U_i\}$ of $X$ such that 
each $(U_i,\Delta|_{U_i})$ is \textit{sharply $F$-pure}. 
The strong $F$-regularity of $(X,\Delta)$ is defined in the same way. 
When $\Delta=0$, we simply say that $X$ is $F$-pure or strongly $F$-regular. 
\end{defn}

Strongly $F$-regular pairs are known to satisfy some properties similar to those of klt pairs. 
The same is true of the relationship between sharply $F$-pure and lc pairs. 
For this reason, our theorems in characteristic zero and $p>0$ can be proved simultaneously. 

\subsection{Weak positivity}\label{subsection:wp}
In this subsection, we recall the notion of the weak positivity of coherent sheaves, 
which was originally introduced by Viehweg \cite[Definition~1.2]{viehweg}. 
Our definition is slightly different from the original. 
\begin{defn} \label{defn:positivity} 
Let $Y$ be a quasi-projective normal variety over a field $k$, 
let $\mathcal{G}$ be a coherent sheaf on $Y$, and let $H$ be an ample Cartier divisor. 
\begin{itemize} 
\item[$($i$)$] 
We say that $\mathcal{G}$ is \textit{generically globally generated} 
if the natural morphism $H^0(Y,\mathcal{G})\otimes_k\mathcal{O}_Y\to\mathcal{G}$ is surjective over the generic point of $Y$.
\item[$($ii$)$] 
We say that $\mathcal{G}$ is \textit{weakly positive} (or \textit{pseudo-effective}) 
if for every $\alpha\in\mathbb Z_{>0}$, there exists some $\beta\in\mathbb Z_{>0}$ such that 
$(S^{\alpha\beta}\mathcal{G})^{**}\otimes\O_Y(\beta H)$ is generically globally generated. 
Here $S^{\alpha\beta}(\underline{\quad})$ and $(\underline{\quad})^{**}$ denote 
the $\alpha\beta$-th symmetric product and the double dual, respectively.  Note that the weak positivity of $\mathcal{G}$ dose not depend on the choice of ample divisor $H$.
\end{itemize}
\end{defn}
The weak positivity of $\mathcal G$ in the sense of the above definition 
is actually equivalent to the pseudo-effectivity in the sense of \cite[Definition~5.1]{BKKMSU15} 
when $\mathcal G$ is a vector bundle, 
and is equivalent to the usual pseudo-effectivity when $\mathcal G$ is a line bundle. 

We need Lemmas~\ref{lem:wp_pullback} and~\ref{lem:wp_gensurj} below 
in order to prove Theorem~\ref{thm:sp}. 
\begin{lem} \label{lem:wp_pullback}
Let $f:Y'\to Y$ be a projective surjective morphism 
between geometrically normal quasi-projective varieties over a field. 
Let $\mathcal G$ be a torsion-free coherent sheaf on $Y$. 
\begin{itemize}
\item[(1)] $($\cite[Lemma~2.15, 1]{Vie95}$)$ Suppose that there are no $f$-exceptional divisors on $Y'$. 
If $\mathcal G$ is weakly positive, then so is $f^*\mathcal G$. 
\item[(2)] Let $E$ be an effective $f$-exceptional Weil divisor on $\tilde Y$. If $\mathcal O_{Y'}(E) \otimes f^*\mathcal G$ is weakly positive, then so is $\mathcal G$. 
\end{itemize}
\end{lem}
\begin{proof}
We first prove (1). 
Note that if $\mathcal G$ is locally free, 
the assertion is proved in the same way as 
in~\cite[Lemma~2.15]{Vie95}. 
Let $V \subseteq Y$ be the maximum open subset such that 
$\mathcal G|_{V}$ is locally free. 
Then the sheaf $\left( f|_{f^{-1}(V)} \right)^*(\mathcal G|_V)$ is weakly positive 
and is isomorphic to $(f^*\mathcal G)|_{f^{-1}(V)}$. 
We then see that $f^*\mathcal G$ is weakly positive, 
since $\mathrm{codim}_{Y'}(Y'\setminus f^{-1}(V)) \ge 2$ by the assumption. 
Next, we show (2). 
We first consider the case when $E=0$ and $\mathcal G$ is locally free. 
If $f$ is generically finite, 
then the assertion is proved in the same way as 
\cite[Lemma~2.15, 2]{Vie95}. 
Otherwise, we need to recall some definition in~\cite{BKKMSU15}. 
For a vector bundle $\mathcal E$ on a quasi-projective variety $X$, 
we define $\mathbf{Bs}(\mathcal E)$ to be the set of points $x \in X$ 
such that the stalk $\mathcal E_{x}$ is \textit{not} generated by global sections of $\mathcal E$. 
Put
$
\mathbf B_-(\mathcal E)
:=\bigcup_{\alpha \ge1}\bigcap_{\beta \ge1} \mathbf{Bs}\left(
\mathcal (S^{\alpha\beta}\mathcal E) \otimes \mathcal H^{\beta}
\right),
$ 
where $\mathcal H$ is an ample line bundle on $X$. 
Note that $\mathbf B_-(\mathcal E)$ is independent of the choice of $\mathcal H$
and is a union of countably many closed subsets of $X$. 
In accordance with~\cite{BKKMSU15}, we use the terminology ``pseudo-effective'' instead of ``weakly positive''. 
Since $f^* \mathcal G$ is pseudo-effective, $\mathbf B_-(f^* \mathcal G) \ne Y'.$ 
Extending the base field, we may assume that it is uncountable. 
Let $Z$ be an intersection of $r$ very general hyperplanes on $Y'$, 
where $r:=\dim Y'-\dim Y$. 
Then, $Z \not\subseteq \mathbf B_-(f^* \mathcal G)$ and 
$f|_Z:Z\to Y$ is a generically finite surjective morphism.
One can easily check that 
$\mathbf B_-((f^*\mathcal G)|_Z) \ne Z,$ 
so $(f^*\mathcal G)|_Z \cong (f|_Z)^*\mathcal G$ is pseudo-effective, 
and hence so is $\mathcal G$. 
Finally, we treat the general case. 
Let $V$ be the maximum open subset 
of $Y \setminus f(\mathrm{Supp}\,E)$ 
such that $\mathcal G|_V$ is locally free. 
Then $E|_{f^{-1}(V)}=0$, so 
$
\left( {f|_{f^{-1}(V)}} \right)^* (\mathcal G|_{V})
\cong 
(\mathcal O_{Y'}(E) \otimes f^*\mathcal G)|_{f^{-1}(V)} 
$ 
is pseudo-effective, and hence so is $\mathcal G|_V$. 
Since $\mathrm{codim}_{Y}(Y\setminus V) \ge 2$, 
we see that $\mathcal G$ is also pseudo-effective. 
\end{proof}
The next lemma follows directly from the definition of weak positivity. 
\begin{lem}[\textup{\cite[Lemma~2.16, c)]{Vie95}}] \label{lem:wp_gensurj}
Let $\tau:\mathcal F\to\mathcal G$ be a generically surjective morphism between coherent sheaves 
on a normal quasi-projective variety over a field. 
If $\mathcal F$ is weakly positive, then so is $\mathcal G$. 
\end{lem}
\section{A version of weak positivity theorem}\label{section:versionsp}
In this section, we give a version of weak positivity theorem for the direct image sheaves of log pluricanonical bundles. 
This is a generalisation of \cite[Main Theorem]{CZ13} and \cite[Theorem 1.3 (1)]{Eji16}. 
In Sections \ref{section:main} and \ref{section:application}, 
we give applications of this weak positivity theorem.  
\begin{thm} \label{thm:sp} 
Let $f:X \to Y$ be a projective morphism between normal projective varieties 
over an algebraically closed field $k$ of characteristic zero $($resp. $p>0$$)$ 
such that $f_*\mathcal{O}_X=\mathcal{O}_Y$ and $K_Y$ is $\mathbb Q$-Cartier.
Let $\Delta=\Delta^{+} -\Delta^{-}$ be a $\mathbb Q$-Weil $($resp. $\mathbb Z_{(p)}$-Weil$)$ divisor 
with the decomposition by the effective $\mathbb{Q}$-divisors, and let $D$ be a $\mathbb Q$-Cartier divisor on $Y$. 
Suppose that $(Z,\Delta^{+}_Z)$ is lc $($resp. sharply $F$-pure$)$ for a general fibre $Z$ of $f$, 
where $K_Z+\Delta^{+}_Z=(K_X+\Delta^+)|_{Z}$, 
and suppose that $\mathrm{Supp}\,\Delta^{-}$ dose not dominate $Y$. 
Moreover assume that $-(K_X+\Delta+f^*D)$ is a nef $\mathbb Q$-Cartier $($resp. $\mathbb Z_{(p)}$-Cartier$)$ divisor. 
Fix an integer $l>0$ such that 
$l(K_X+\Delta)$ and $l(K_Y+D)$ are Cartier and $l\Delta^-$ is integral. 
Then, there exists an effective $f$-exceptional divisor $B$ on $X$ such that 
$$\mathcal O_X\big( l(-f^*(K_Y+D)+\Delta^{-}+B) \big)$$ 
is a weakly positive sheaf. 
Furthermore, this $B$ can be substituted by $0$ 
if $Y$ has only canonical singularities.
\end{thm}
\begin{proof}
We first prove the assertion under the assumption that $f$ is equi-dimensional. 
Note that, in this case, the pull-back of a Weil divisor by $f$ is well-defined, 
and coincides with the usual one if the divisor is $\mathbb Q$-Cartier. 
For this reason, in this step, we need not assume that $K_Y$ is $\mathbb Q$-Cartier.
Set $\mathcal F:=\mathcal O_X\big( l(-f^*(K_Y+D)+\Delta^-) \big)$. 
Fix $0<\alpha\in\mathbb Z$ (resp. $0<\alpha\in\mathbb Z\setminus p\mathbb Z$) and an ample Cartier divisor $A$ on $X$. 
It suffice to show that there is some $\beta\in\mathbb Z_{>0}$ 
such that $\mathcal F^{[\alpha\beta]}\otimes\mathcal O_X(\beta lA)$ is weakly positive, 
where $\mathcal F^{[n]}:=(\mathcal F^{\otimes n})^{**}$. 
Since $-(K_X+\Delta+f^*D)+\alpha^{-1}A$ is an ample $\mathbb Q$-Cartier (resp. $\mathbb Z_{(p)}$-Cartier) divisor, 
by \cite[Corollary 6.10]{SZ13}, we can take a $\mathbb Q$-Cartier (resp. $\mathbb Z_{(p)}$-Cartier) divisor $\Gamma\ge0$ on $X$ 
such that $\Gamma\sim_{\mathbb Q}-(K_X+\Delta+f^*D)+\alpha^{-1}A$ and $(Z,\Delta^{+}_Z+\Gamma|_Z)$ is lc (resp. sharply $F$-pure). 
Note that, in characteristic $p>0$, 
we see that $(X_{\overline\eta},(\Delta^{+}+\Gamma)|_{X_{\overline\eta}})$ is also sharply $F$-pure by \cite[Corollary 3.31]{PSZ13}, 
where $X_{\overline\eta}$ is the geometric generic fibre of $f$. 
Applying \cite[Theorem~4.13]{Cam04} or \cite[Theorem~1.1]{Fuj13} (resp. \cite[Theorem~5.1]{Eji15}), we see that 
$$\bigl(f_*\mathcal O_X(lm(K_{X}+\Delta^{+}+\Gamma))\bigr)\otimes\mathcal O_Y(-lmK_Y)$$ 
is weakly positive for any sufficiently divisible $m \in\mathbb Z_{>0}$. 
Note that, in characteristic $p>0$, since we have 
$$
(K_X+\Delta^{+}+\Gamma)|_{X_{\overline\eta}}
\sim_{\mathbb Q}(\Delta^{-}+f^*D+\alpha^{-1}A)|_{X_{\overline\eta}}
\sim_{\mathbb Q}\alpha^{-1}A|_{X_{\overline\eta}},
$$ 
the hypothesis of \cite[Theorem~5.1]{Eji15} is satisfied as shown in \cite[Example~3.11]{Eji15}. 
Take an integer $\beta\gg0$ that is sufficiently divisible. 
Then, we see that $\mathcal F^{[\alpha\beta]}\otimes\mathcal O_X(\beta l A)$ is weakly positive,
combining Lemmas~\ref{lem:wp_pullback} and~\ref{lem:wp_gensurj} with the following sequence of generically surjective morphisms:
\begin{align*}
& f^*\biggl(\bigl(f_*\mathcal O_X(\alpha\beta l(K_X+\Delta^{+}+\Gamma))\bigr)\otimes\mathcal O_Y(-\alpha\beta lK_Y)\biggr) \\
\to & \mathcal O_X(\alpha\beta l(K_X+\Delta^{+}+\Gamma))\otimes f^*\mathcal O_Y(-\alpha\beta lK_Y)  
\hspace{57pt} \textup{\footnotesize by $f^*f_*(\underline{\quad})\to(\underline{\quad})$} \\
\cong & \mathcal O_X(\alpha\beta l(\Delta^{-}-f^*D+\alpha^{-1}A))\otimes f^*\mathcal O_Y(-\alpha\beta lK_Y)  
\hspace{33pt} \textup{\footnotesize by definition of $\Gamma$} \\
\textup{$
\to
$}
&
\textup{$
\mathcal O_X \big(\alpha\beta l(\Delta^{-}-f^*(K_Y+D)+\alpha^{-1}A) \big) 
$
\hspace{80pt} {\footnotesize $(\underline{\quad}) \to (\underline{\quad})^{**}$ } 
} \\
\textup{$
\cong 
$} 
&
\textup{$
\mathcal F^{[\alpha\beta]}\otimes\mathcal O_X(\beta l A) 
$} 
\hspace{193pt} \textup{\footnotesize by definition of $\mathcal F$.}
\end{align*}

Next, we consider the general case. 
By the flattening theorem \cite[3.3. The flattening lemma]{AO00}, 
we have a birational morphism $\sigma:\tilde Y \to Y$ from a normal projective variety $\tilde Y$ 
such that $\tilde f:\tilde X \to \tilde Y$ is equi-dimensional, 
where $\tilde f$ is the natural morphism from 
the normalisation $\tilde X$ of the main component of $\tilde Y \times_{Y} X$ 
to $\tilde Y$.
We have the following commutative diagram:
$$
\xymatrix@R=20pt@C=20pt{ \tilde X \ar[r]^{\rho} \ar[d]_{\tilde f} & X \ar[d]^{f} \\ 
\tilde Y \ar[r]^{\sigma} & Y.}
$$
Let $E \ge 0$ be a $\sigma$-exceptional divisor on $\tilde Y$ 
such that $-K_{\tilde Y} \le -\sigma^*K_Y +E$. 
Let $\tilde \Delta$ be a $\mathbb Q$-Weil divisor on $\tilde X$ 
such that $K_{\tilde X}+\tilde \Delta=\rho^*(K_X+\Delta)$. 
Then, $-(K_{\tilde X} + \tilde \Delta + {\tilde f}^*(\sigma^*D))$ is 
equal to the nef divisor $\rho^* \big( -(K_X +\Delta +f^*D) \big)$, 
so it follows from the previous step that 
$
\mathcal O_{\tilde X}\big( l(
-{\tilde f}^* (K_{\tilde Y} +\sigma^*D) +\tilde\Delta^{-} 
)\big)
$
is weakly positive. 
Here, $\tilde\Delta^-$ is defined by the natural decomposition $\tilde \Delta =\tilde \Delta^+ -\tilde\Delta^-$
by the effective $\mathbb Q$-divisors. 
Since 
$
-{\tilde f}^* (K_{\tilde Y}+\sigma^*D) 
\le -\rho^*f^*(K_Y+D) +{\tilde f}^*E, 
$
the sheaf
$
\mathcal O_{\tilde X}\big( l(
-\rho^*f^*(K_Y+D) +\tilde\Delta^{-} +{\tilde f}^*E
)\big)
$
is also weakly positive, and hence so is 
$
\mathcal O_X\big( l(
-f^*(K_Y+D) +\Delta^{-} +\rho_*({\tilde f}^*E)
)\big), 
$
where $\rho_*\tilde \Delta^{-}=\Delta^{-}$.  
Set $B:=\rho_*({\tilde f}^*E)$. Then $f_*B=\sigma_*E=0$. 
If $Y$ is canonical, we can put $E=0$, and then $B=0$. 
This completes the proof. 
\end{proof}
When $\Delta^-=0$, Lemma~\ref{lem:wp_pullback}~(2) tells us that $-(K_Y+D)$ is pseudo-effective. 
%
In particular, one can recover \cite[the Main Theorem]{CZ13} and \cite[Theorem 1.3 (1)]{Eji16}. 
\begin{rem}\label{rem:cz}
(1) Our proof differs from that in \cite{CZ13}, even when $X$ and $Y$ are smooth and both $\Delta$ and $D$ are zero. 
We explain briefly the difference. 
In the proof of Theorem~\ref{thm:sp}, we use the natural morphism 
$$ f^*\biggl(\bigl(f_*\O_X(\alpha\beta(K_X+\Gamma))\bigr)\otimes\O_Y(-\alpha\beta K_Y)\biggr) 
\to \mathcal O_X(\beta(-\alpha f^*K_Y+A))$$
to derive the weak positivity of the target from that of the source. 
The proof in~\cite{CZ13} employs another morphism, which is induced by a trace map: 
\begin{align*}
&S^n\biggl(\bigl(f_*\mathcal O_X(\alpha\beta(K_X+\Gamma))\bigr) \otimes\omega_Y^{-\alpha\beta}\biggr)
\otimes
S^n\biggl(\bigl(f_*\mathcal O_X(\alpha\beta(K_X+\Gamma))\bigr)^*\otimes\mathcal O_Y(\beta L)\biggr) \\
&\to\mathcal O_Y(\beta n(-\alpha K_Y+L)).
\end{align*}
Here, $L$ is an ample divisor on $Y$, which is needed to get the weak positivity of the source. 

\noindent(2) Next, we give a supplement of the proof of \cite[Main Theorem]{CZ13}. 
Suppose for simplicity that $X$ and $Y$ are smooth and $D$ is zero in the notation of \cite[Main Theorem]{CZ13}. 
Then, $f'$ and $\pi:X'\to X$ in \cite{CZ13} are coincide with $f$ and $\mathrm{id}:X\to X$, respectively. 
Take $\delta,\epsilon\in\mathbb Q_{>0}$ with $\epsilon/\delta\ll1$. 
In \cite[Page~1855, Line~12 of Proof of the Main Theroem]{CZ13}, we could not get the global generation of 
$$
\O_Y(m\delta L)\otimes(f'_*\O_{X'}(\pi^*m\epsilon A))^*
$$ 
for any sufficiently divisible $m$. 
But, if we fix some $k$, we can get the global generation for any $m$ with $m \epsilon < k$. 
Thanks to this fact, we only need to make a small modification. 
For example, this is done as follows. 
Let $A$ be sufficiently ample. 
Considering Castelnuovo--Mumford regularity together with Fujita's vanishing theorem, 
we see that $A+N$ is free for any nef divisor $N$ on $X$. 
Fix $n\in\Z_{>0}$. 
Since $A-nK_X$ is free, there is an $E\in|A-nK_X|$ such that $(X,\frac{1}{n}E)$ is lc. 
Then, 
\begin{align*}
(f_*\O_X(A))\otimes \omega_Y^{-n}
\cong f_*\O_X(A-nf^*K_Y)
\cong f_*\O_X(nK_{X/Y}+E),
\end{align*}
so we find this sheaf is weakly positive by \cite{Cam04} or \cite{Fuj17}. 
%
Take $m_0\in\Z_{>0}$ such that 
$
\O_Y(m_0L)\otimes(f_*\O_X(A))^*
$ 
is globally generated. 
Then, a non-trivial trace map
\begin{align*}
\big(\O_Y(m_0L)\otimes(f_*\O_X(A))^*\big)\otimes
\big((f_*\O_X(A))\otimes\omega_Y^{-n}\big)
\to \O_Y(m_0L)\otimes\o_Y^{-n}
\end{align*}
implies $-K_Y+\frac{m_0}{n}L$ is pseudo-effective, 
and hence so is $-K_Y=\underset{n\to\infty}{\mathrm{lim}}(-K_Y+\frac{m_0}{n}L)$. Therefore we complete the proof.
\end{rem}
\section{Injectivity for restrictions on  fibres}\label{section:main} 
In this section, we discuss the restriction map of global sections of some line bundles 
(basically which are relative anti-log canonical bundles) 
under some semi-positivity assumptions. 
Throughout this section, we follow the notation below:
\begin{notation} \label{notation:main}
Let $f:X \dashrightarrow Y$ be an almost holomorphic dominant rational map between normal projective varieties 
over an algebraically closed field $k$ of characteristic zero $($resp. $p>0$$)$, 
i.e. $f$ is a surjective regular morphism over some non-empty open set of $Y$.
Let $\Delta$ be an effective $\mathbb Q$-Weil $($resp. $\mathbb Z_{(p)}$-Weil$)$ divisor on $X$.  
Let $Y_0$ be the maximal open set that $f$ is regular on and let $X_0=f^{-1}(Y_0)$.  
Suppose that the following conditions hold:
\begin{itemize}
\item[{\rm(1)}] $K_X+\Delta$ is $\Q$-Cartier;
\item[{\rm(2)}] a general fibre $F$ of $f$ is normal and $(F,\Delta_F)$ is klt $($resp. strongly $F$-regular$)$, 
where $K_F+\Delta_F=(K_X+\Delta)|_{F}$; 
\item[{\rm(3)}] $Y$ has only canonical singularities;
\item[{\rm(4)}] $(f|_{X_0})_*\mathcal{O}_{X_0}=\mathcal{O}_{Y_0}.$
\end{itemize}
Moreover let $\pi: W \to X $ be a resolution of the indeterminacy locus of $f$ with a normal variety $W$ and $g: W \to Y$ the induced morphism. 
\end{notation}
We now have the following diagram:
\begin{equation*}
\xymatrix{ &W \ar[dl]_{\pi} \ar[dr]^{g}\\
 X \ar@{-->}[rr]^{f}  & & Y.}
\end{equation*} 
The aim of this section is to prove the following theorem. 
\begin{thm} \label{thm:inj} 
With the notation above, let $Z:=f^{-1}(y)$ be a closed fibre of $f$ over a regular point $y\in Y_0$. 
Suppose further that the following conditions hold:
\begin{itemize}
\item[{\rm (i)}] $L:=-(\pi^*(K_{X}+\Delta)-g^*K_Y)$ is nef;
\item[{\rm (ii)}] $f$ is flat at every point in $Z$; 
\item[{\rm (iii)}] $\mathrm{Supp}\,\Delta$ does not contain $Z$; 
\item[{\rm (iv)}] $Z$ is normal. 
\end{itemize}
Then the support of any effective $\Q$-Cartier divisor $\Gamma$ on $W$ with $\Gamma\equiv L$ 
does not contain $\pi^{-1}(Z)$. 
\end{thm}
Before proving the theorem, let us consider the case of ruled surfaces. 
\begin{eg}\label{eg:ruled surf} 
Let $Y$ be a smooth projective curve and $D$ a Cartier divisor on $Y$ with $\deg D\le0$. 
Let $f:X:=\P(\O_Y\oplus\O_Y(D))\to Y$ be the projective bundle.
Set $C_0:=\P(\O_Y)\subset X$. 
As shown in  \cite[\S5.2]{Har77}, we have $C_0^2=\deg D$ and $-K_{X/Y}\sim 2C_0-f^*D$. 
Suppose now that $-K_{X/Y}$ is nef. 
Then $$0\le-K_{X/Y}\cdot C_0=2C_0^2-f^*D\cdot C_0=\deg D\le0,$$ 
and so $-K_{X/Y}\cdot C_0=C_0^2=0$.
Take an effective $\Q$-divisor $\Gamma$ on $X$ with $\Gamma\equiv -K_{X/Y}$. 
Since $\Gamma\cdot C_0=0$, we conclude that $\mathrm{Supp}\,\Gamma$ does not contain a fibre of $f$. 
\end{eg}
By applying our weak positivity theorem (Theorem~\ref{thm:sp}), 
we obtain the following proposition:

\begin{prop}\label{prop:p.e.} 
With the same notation as \ref{notation:main}, let $D$ and $E$ be $\Q$-Cartier divisors on $Y$. 
Suppose that the following conditions hold:
\begin{itemize}
\item[{\rm (a)}] $L:=-(\pi^*(K_{X}+\Delta) +g^*D)$ is nef;
\item[{\rm (b)}] there exists a $\mathbb Q$-Cartier divisor $\Gamma\ge0$ on $W$ with $\Gamma\equiv L-g^*E$. 
\end{itemize}
Then there exists some $\e\in\mathbb Q_{>0}$ such that 
$\O_X(-m\pi_*g^*(K_Y+D+\e E))$ is weakly positive for any sufficiently divisible $m\in\Z_{>0}$. 
\end{prop}
\begin{proof}
We take effective $\mathbb{Q}$-divisors $\Delta_W$ and $G$ on $W$ such that 
they have no common components and 
$$
\pi^*(K_X+\Delta)+G=K_W+\Delta_W.
$$
Then $G$ is $\pi$-exceptional. 
For any $\e\in[0,1)\cap \mathbb Q$, we fix the notation as follows: 
$$\Delta_W^{(\e)}:=\Delta_W+\e\Gamma;~D^{(\e)}:=D+\e E;~L^{(\e)}:=-(K_W+\Delta_W^{(\e)}-G+g^*D^{(\e)}).$$ 
We then have 
\begin{align*}
L^{(\e)}&=-(K_W+\Delta_W+\e\Gamma-G+g^*(D+\e E)) \\
&=-(K_W+\Delta_W-G+g^*D)-\e(\Gamma+g^*E) \\
&\equiv L-\e L,
\end{align*}
so (a) indicates $L^{(\e)}$ is nef.
Take a general fibre $F$ of $f$. 
Note that now $F$ is also a general fibre of $g$.
Since $(F,\Delta_F)$ is klt (resp. strongly $F$-regular), so is $(F,\Delta^{(\e)}_F)$ for $\e\ll1$, 
where $K_F+\Delta^{(\e)}_F=(K_W+\Delta_W^{(\e)})|_{F}$. 
Fix such $\e$. 
Applying 
Theorem~\ref{thm:sp} to $g,W,Y,\Delta_W^{(\e)}-G$ and $D^{(\e)}$, 
we get that $\O_W\big(m(-g^*(K_Y+D^{(\e)})+G)\big)$ is weakly positive for any sufficiently divisible $m\in\Z_{>0}$.  
Thus, the sheaf 
$$
\big(\pi_*\O_W\big(m(-g^*(K_Y+D^{(\e)})+G)\big)\big)^{**}\cong\O_X(-m\pi_*g^*(K_Y+D^{(\e)}))
$$ 
is also weakly positive, which is our assertion. 
\end{proof}
The next corollary follows directly from the proposition.
\begin{cor}\label{cor:exc} 
With the same notation as \ref{notation:main}, suppose that $L:=-(\pi^*(K_{X}+\Delta)-g^*K_Y)$ is nef. 
Let $E$ be a nonzero effective $\mathbb Q$-Cartier divisor on $Y$. 
If $L-g^*E$ is numerically equivalent to an effective $\mathbb Q$-Cartier divisor on $W$, 
then $g^*E$ is $\pi$-exceptional. 
\end{cor}
\begin{proof}
Fix an ample Cartier divisor $A$ on $X$.
Applying Proposition~\ref{prop:p.e.} with $D=-K_Y$, 
we see that $\mathcal O_X(-m\pi_*g^*E )$ is weakly positive, 
for some sufficiently divisible $m\in\mathbb Z_{>0}$. 
Set $P:=mg^*E \ge0$. We show $\pi_*P=0$. 
Fix $\alpha \in \mathbb Z_{>0}$. 
Put $M:=-\alpha P +\pi^*A$. 
The weak positivity of $\mathcal O_X(-\pi_* P)$ implies that 
$
\pi_*M
= -\alpha \pi_*P +A 
\sim_{\mathbb Q} \Gamma \ge 0
$
for a $\mathbb Q$-Weil divisor $\Gamma$ on $X$. 
Thanks to \cite{deJ96}, 
we have a regular alteration $\rho:\tilde W \to W$ of $W$, 
i.e. a generically finite morphism 
from a smooth projective variety $\tilde W$ to $W$. 
Set $\tilde \pi:\tilde W\xrightarrow{\rho} W \xrightarrow{\pi} X$
and $U:=\tilde \pi^{-1}(X_{\mathrm{sm}}) \subseteq \tilde W$. 
Since 
$
(\rho^*M) |_{U}
\sim_{\mathbb Q}
\left( \tilde \pi|_{U} \right)^* \left( \Gamma|_{X_{\mathrm{sm}}} \right)
\ge 0,  
$
there is a divisor $B\ge0$ supported on $\tilde W \setminus U$ such that 
$ \rho^*M +B $
is $\mathbb Q$-linearly equivalent to an effective $\mathbb Q$-divisor on $\tilde W$. 
Note that $\tilde \pi_* B=0$.
Then 
$$
M \cdot (\pi^* A)^{\dim X-1}
=\frac{(\rho^*M) \cdot (\tilde \pi^* A)^{\dim X-1}}{\deg \rho}
=\frac{(\rho^*M +B) \cdot (\tilde \pi^* A)^{\dim X-1}}{\deg \rho}
\ge 0, 
$$
and so 
$
-\alpha P \cdot (\pi^*A)^{\dim X-1} +(\pi^*A)^{\dim X} 
\ge 0. 
$ 
This means $P \cdot (\pi^*A)^{\dim X-1}=0$, 
since $\alpha$ can be any positive integer. Hence $\pi_*P=0$ as desired.
\end{proof}
Using this corollary, we now prove the main theorem. 
We start with a brief explanation of the strategy. 
Suppose for simplicity that $\dim Y=1$. Then, we can regard $y\in Y_0$ as a divisor. 
Let $\Gamma$ be a $\mathbb Q$-Cartier divisor on $X$ with $0\le \Gamma\equiv L$. 
Let $b\ge0$ be the coefficient of $B:=g^*y$ in $\Gamma$. 
We then have $L-g^*(ay)\equiv\Gamma-bB\ge0$. 
The assertion is equivalent to saying that $b=0$. 
If $b>0$, then Corollary~\ref{cor:exc} implies $B$ is $\pi$-exceptional, a contradiction. 

In the case when $\dim Y\ge 2$, we need to consider the blowing-up of $Y$ at $y$. 
\begin{proof}[Proof of Theorem~\ref{thm:inj}.]
Let $\tilde W$ (resp. $\tilde X$ and $\tilde Y$) be the blowing-up of $W$ (resp. $X$ and $Y$) along $g^{-1}(y)$ (resp. $Z$ and $y$), 
let $\varphi_W$ (resp. $\varphi_X$ and $\varphi_Y$) be the natural projection, 
and let $B$ (resp. $\tilde Z$ and $E$) denote its exceptional locus. 
We now have the following commutative diagram: 
\begin{align*}
\xymatrix@R=25pt@C=25pt{ 
\tilde Z \ar@{_(->}[d]_{} 
& B \ar[l]_{\cong} \ar@{_(->}[d]^{} \ar[r]^{} & E \ar@{_(->}[d]^{} \\
\tilde X \ar[d]_{\varphi_X} & \tilde W \ar[d]^{\varphi_W} \ar[l]_{\tilde \pi} \ar[r]^{\tilde g} & \tilde Y \ar[d]^{\varphi_Y} \\ 
X & W \ar[l]_{\pi} \ar[r]^{g} & Y. }
\end{align*}
Since $\pi$ and $g$ are flat at every point in $g^{-1}(y)$,
each square in the above diagram is cartesian. 
Then $B={\tilde g}^*E$ as divisors and $B\cong g^{-1}(y)\times E\cong Z\times E$. 
Now, we see that $\tilde W$ is normal, because $\tilde g$ is a flat morphism with normal fibres in a neighbourhood of $B$. 
Hence $\tilde X$ is also normal. 
Set $\tilde \Delta:={\varphi_X}^{-1}_*\Delta$ and  
$\tilde L:=-{\tilde \pi}^*(K_{\tilde X}+\tilde\Delta)+{\tilde g}^*K_{\tilde Y}$. 
Then $\tilde L$ is nef. 
Indeed, one can easily check that 
$$
\tilde \pi ^*(K_{\tilde X}+\tilde \Delta) -{\tilde g}^*K_{\tilde Y}
={\varphi_W}^*(\pi^*(K_X+\Delta)-g^*K_Y), 
$$
which means 
$
\tilde L={\varphi_W}^*L. 
$
Let $\Gamma$ be a $\mathbb Q$-Cartier divisor on $W$ with $0\le\Gamma\equiv L$, 
and let $b\ge0$ be the coefficient of $B$ in ${\varphi_W}^*\Gamma$. 
Then $\tilde L-{\tilde g}^*(bE) \equiv {\varphi_W}^*\Gamma-bB\ge0$. 
Our claim is equivalent to saying that $b=0$. 
If $b>0$, then Corollary~\ref{cor:exc} implies $B$ is $\tilde\pi$-exceptional, a contradiction. 
\end{proof}
We can use Theorem \ref{thm:inj} to study the section rings of nef relative anti-canonical divisors. 
\begin{defn}
Let $D$ be a $\Q$-Weil divisor on a normal variety. 
The \textit{section ring of $D$} is given by 
$$R(X,D):=\bigoplus_{m\ge0}H^0(X,\lfloor mD\rfloor).$$
\end{defn}
\begin{cor}\label{cor:anti_can_ring} 
In the same situation as in Theorem \ref{thm:inj}, 
the ring homomorphism 
$$R(X,L)\to R(Z,L|_Z)$$ 
induced by restriction is injective. 
\end{cor}
\begin{proof}
Take an $m\in\Z_{>0}$ and an $s\in H^0(X,\lfloor mL\rfloor)$ with $s|_Z=0$. 
Then $s^d|_Z=(s|_Z)^d=0$ for some $d\in\Z_{>0}$ sufficiently divisible. 
Since $s^d\in H^0(X,-dmL)$, Theorem~\ref{thm:inj} implies $s^d=0$, and so $s=0$. 
\end{proof}
\section{Main theorems}\label{section:application} 
In this section, we discuss applications of Theorem \ref{thm:inj}. 
Throughout this section, we use Notation \ref{notation:main} and suppose that $L:=-\pi^*(K_{X}+\Delta)+g^*K_Y$ is nef.
Take a general fibre $F$ of $f$. 
Corollary~\ref{cor:anti_can_ring} implies that
$$\kappa(X,L)\le\kappa(F,L|_F).$$
Combining this with the inequality 
$$\kappa(F,L|_F)\le\dim F=\dim X-\dim Y,$$
we obtain the following:
\begin{cor}\label{cor:inequality} 
In this situation, the inequality 
$$\kappa(X,L)\le\dim X-\dim Y$$ holds true. 
In particular, if $L$ is nef and big, then $\dim Y=0$. 
\end{cor}
The next corollary gives an affirmative answer to  a question of Hacon--Mckernan:
\begin{cor}\label{cor:hm-inequality}
Let $(X,\Delta)$ be a projective klt pair over a field in characteristic zero 
and $r: X \dashrightarrow W$ the maximal rationally chain connected fibration.
Suppose that $-(K_X+\Delta)$ is nef. 
Then $\kappa(X, -(K_X+\Delta)) \leq \kappa (F, -(K_F+\Delta_F))$. 
Here, $F$ is a general fibre of $r$ and $K_F+\Delta_F=(K_X+\Delta)|_{F}$. 
In particular, it holds that 
$$ \kappa(X, -(K_X+\Delta)) \leq \mathrm{dim}\,X-\mathrm{dim}\,W.
$$
\end{cor}
\begin{proof}By \cite[Corollary 1.5]{HM07}, $r $ is the maximal rational connected fibration.  First  we may assume that $W$ is smooth since $r$ is a rational map. \cite[Corollary 1.4]{ghs} implies that $W$ is not uniruled.   By \cite[Remark 1]{Zha05}, we see that the numerical Kodaira dimension  $\kappa_{\sigma}(K_W)$ is zero. Thus we have a good minimal model of $W$ by \cite{druel} and \cite[V, 4.9. Corollary]{nak}. Thus we may assume that $W$ has only $\mathbb {Q}$-factorial terminal singularities and $K_W \sim_{\mathbb{Q}}0$. Thus the desired inequality follows from Corollary \ref{cor:inequality}.
\end{proof}
\section{Questions} \label{section:questions}
Surprisingly, it is proved in \cite{CH17} that 
the maximal rationally connected fibration of a projective manifold with nef anti-canonical bundle 
can be taken as a regular morphism. 
In view of this result, it seems to be natural to ask the following question:
\begin{ques}\label{intro:regularity}
Let $(X,\Delta)$ be a projective (even compact K\"ahler) klt pair over the complex number field such that $-(K_X+\Delta)$ is nef, 
and let $r: X \dashrightarrow W$ be the maximal rationally chain connected fibration. 
Can $r$ be represented by a regular morphism?
\end{ques}
We also should ask for the K\"ahler case of Theorem \ref{thm-intro:hm-inequality}:
\begin{ques}\label{intro:kahler}
Let $(X,\Delta)$ be a compact K\"ahler  klt pair
and $r: X \dashrightarrow W$ the maximal rationally chain connected fibration.
Suppose that $-(K_X+\Delta)$ is nef. Does the inequality 
$$
\kappa(X, -(K_X+\Delta)) \leq \mathrm{dim}\,X-\mathrm{dim}\,W 
$$
hold true? 
\end{ques}
\bibliographystyle{abbrv}
\bibliography{ref_all}
\end{document}